\documentclass[11pt]{article}

	\usepackage{mathtools}
	\usepackage{amssymb}
	\usepackage{amsthm}
	\usepackage{amsmath}
	\usepackage{mathrsfs}
	\usepackage{tikz}
	\usepackage{tikz-cd}
		\usetikzlibrary{decorations.pathmorphing}
	\usepackage{xcolor} 
	\usepackage{stmaryrd}
	\usepackage{bbm}
	\usepackage{adjustbox}	
	\usepackage{setspace}
	\usepackage{geometry}
	\usepackage{array} 
		\newlength\mylen
		\newcolumntype{C}{>{\hfil$}p{\mylen}<{$\hfil}}
	\usepackage{graphicx}
		\graphicspath{ {./Images/} }
	\usepackage{caption}
	\usepackage{subcaption}

\newcommand{\pair}[1]{\left\langle#1\right\rangle}
\newcommand{\tate}[1]{\left\langle#1\right\rangle}
\newcommand{\ldef}[1]{\textcolor{purple}{\textit{#1}}}

\DeclareMathOperator{\Ind}{Ind}

\DeclareMathOperator{\triv}{triv}
\DeclareMathOperator{\sign}{sign}

\DeclareMathOperator{\Perv}{Perv}
\DeclareMathOperator{\mix}{mix}

\DeclareMathOperator{\BS}{BS}
\DeclareMathOperator{\Ch}{Ch}

\DeclareMathOperator{\R}{\mathbb{R}}

\DeclareMathOperator{\Z}{\mathbb{Z}}

\DeclareMathOperator{\h}{\mathfrak{h}}

\DeclareMathOperator{\cO}{\mathcal{O}}
\DeclareMathOperator{\cT}{\mathcal{T}}

\DeclareMathOperator{\sD}{\mathscr{D}}

\DeclareMathOperator{\tA}{\textbf{A}}

\DeclareMathOperator{\Hom}{Hom}

\DeclareMathOperator{\id}{id}

\DeclareMathOperator{\Sym}{Sym}
\DeclareMathOperator{\alphac}{\check{\alpha}}

\newtheorem{thm}{Theorem}
\newtheorem{cor}[thm]{Corollary}
\newtheorem{lem}[thm]{Lemma}
\newtheorem{prop}[thm]{Proposition}

\theoremstyle{remark}
\newtheorem{rem}[thm]{Remark}

\title{Monotonicity of inverse Kazhdan-Lusztig polynomials}
\author{J. Baine}
\date{}
\begin{document}

	\maketitle
	\begin{abstract}
		\noindent
		For arbitrary Coxeter systems, we prove that inverse Kazhdan-Lusztig polynomials satisfy a monotonicity property. 
		This follows from the validity of Soergel's conjecture and the existence of injective morphisms between Rouquier complexes in the mixed perverse Hecke category. 
		The monotonicity property is generalised to parabolic Kazhdan-Lusztig polynomials. 
	\end{abstract}
	
\section{Introduction}
	\noindent
	Kazhdan-Lusztig polynomials $h_{y,x}$ arise as the change of basis coefficients between the standard and Kazhdan-Lusztig bases of a Hecke algebra.
	Since their introduction it was immediately apparent these polynomials exhibit many remarkable properties.
	Two of the most famous properties they satisfy are:
	\[
	\settowidth\mylen{XXXXXXXXXXXXXXXXXXXXXXXXXXXXXXXXXXXXXXXXX}
	\begin{array}{lC}
		\text{non-negativity:} & h_{y,x} \in \Z_{\geq 0}[v],~~~~~~~~
		\\
		\text{monotonicity:} & h_{y,x} \preceq v^{\ell(z)-\ell(y)}~h_{z,x} 
	\end{array}
	\]
	where $z \leq y \leq z$ in the Bruhat order and $\preceq$ denotes coefficient-wise comparison. 
	Similarly, inverse Kazhdan-Lusztig polynomials $h^{y,x}$ are known to satisfy: 
	\[
	\settowidth\mylen{XXXXXXXXXXXXXXXXXXXXXXXXXXXXXXXXXXXXXXXXX}
	\begin{array}{lC}
		\text{non-negativity:}& h^{y,x} \in \Z_{\geq 0}[v] ~~~~~~~~~~~
	\end{array}
	\]
	and, in the finite case, an analogous monotonicity property. 
	Despite their combinatorial origins, no elementary proofs of these properties are known. 
	Instead a fecund approach has been to study categories that categorify Hecke algebras, and deduce these properties from the structure of objects therein.
	\\
	\par 
	The main result of this note is the following monotonicity property of inverse Kazhdan-Lusztig polynomials, which holds for arbitrary Coxeter systems. 
	\begin{thm}
	\label{Theorem: monotonicity}
		Let $(W,S)$ be an arbitrary Coxeter system. 
		For any $x, y, z \in W$ satisfying \linebreak 
		$z\leq y \leq x$ in the Bruhat order, the inverse Kazhdan-Lusztig polynomials satisfy
		\begin{align*}
			h^{z, y} \preceq v^{\ell(x)- \ell(y)} ~h^{z, x}.
		\end{align*}
		More generally, for any subset $I \subseteq S$ and minimal length coset representatives $x, y, z \in W^I$ satisfying $z\leq y \leq x$ in the induced Bruhat order, the antispherical inverse Kazhdan-Lusztig polynomials satisfy
		\begin{align*}
			 n^{z, y} \preceq v^{\ell(x)-\ell(y)} ~n^{z, x}.
		\end{align*}
	\end{thm} 
	This result can be seen as a combinatorial shadow of the existence of injective morphisms between certain Rouquier complexes in the heart of the perverse $t$-structure on the bounded homotopy category of diagrammatic Soergel bimodules associated to a realisation satisfying Soergel's conjecture.    
	This approach is natural and, in a sense we will soon make clear, brings this chapter of the story of monotonicity in Kazhdan-Lusztig theory full circle. 
	\\
	\par 
	In \cite{Irv88} Irving first showed that Kazhdan-Lusztig polynomials for finite Weyl groups exhibit a monotonicity property. 
	That proof is remarkable as Irving actually proved the corresponding statement for inverse Kazhdan-Lusztig polynomials by studying the socle filtrations of Verma modules in the principal block of category $\cO_0$. 
	He then deduced the monotonicity of classical Kazhdan-Lusztig polynomials using the Kazhdan-Lusztig inversion formula.  
	\\
	\par 
	Since the work of Beilinson, Ginzburg and Soergel in \cite{BGS96}, the paradigm has emerged that Kazhdan-Lusztig theory is most evident when studying $\Z$-graded Hecke categories.
	A $\Z$-graded analogue of $\cO_0$ was constructed in \cite{BGS96} and it was shown that Verma modules admit lifts to this category. 
	Thus, monotonicity follows from the existence of injective morphisms between graded Verma modules. 
	\\
	\par 
	In \cite{Soe92,Soe07} Soergel constructed various categories of (bi)modules associated to Coxeter systems.
	Soergel conjectured that, under various assumptions, certain filtrations of these bimodules are encoded by Kazhdan-Lusztig polynomials. 
	Elias and Williamson proved Soergel's conjecture in \cite{EW14}. 
	When the Coxeter group is a Weyl group, the graded category $\cO$ considered in \cite{BGS96} is equivalent to the heart of the perverse $t$-structure on the bounded homotopy category of Soergel modules (satisfying Soergel's conjecture).  
	Under this equivalence, graded Verma modules correspond to the Rouquier complexes we consider. 
	\\
	\par 
	Thus, Rouquier complexes for arbitrary Coxeter systems are the natural objects to consider if one is interested in inverse Kazhdan-Lusztig polynomials of arbitrary Coxeter systems. 
	Given that the story of monotonicity in Kazhdan-Lusztig theory commenced with Irving's study of Verma modules in \cite{Irv88}, it seems apt that the final families of (inverse) Kazhdan-Lusztig polynomials shown to exhibit monotonicity properties are those arising from natural generalisations of Verma modules. 
	\\
	\par 
	The structure of this note is as follows: 
	In section 1 we fix notation relating to Hecke algebras. 
	In section 2 we introduce various Hecke categories. 
	In section 3 we show that injective morphisms between objects in mixed perverse Hecke categories give rise to families of polynomials satisfying a monotonicity property. 
	In section 4 we introduce Rouquier complexes and prove the existence of injective morphisms between them, and deduce the main theorem. 
	In section 5 we recall parabolic Kazhdan-Lusztig theory and deduce the antispherical result from the non-parabolic setting.  
	
\section{Hecke algebras}	

	We begin by fixing notation related to Hecke algebras, their Kazhdan-Lusztig bases and the associated (inverse) Kazhdan-Lusztig polynomials. 
	\\
	\par 
	Let $(W,S)$ be an arbitrary Coxeter system.
	Denote by $\ell$ the length function on $W$, and $\leq$ the Bruhat order. 
	\\
	\par 
	Associated to the Coxeter system $(W,S)$ is the \ldef{Hecke algebra} $H$. 
	It is the associative $\Z[v,v^{-1}]$-algebra on the symbols $\{ \delta_x ~\vert~ x \in W\}$ subject to the relations:
	\begin{align*}
		(\delta_s+v)(\delta_s-v^{-1}) &=0 && \text{for each } s\in S, \text{ and,}
		\\
		\delta_x \delta_y &= \delta_{xy} && \text{whenever } \ell(x) + \ell(y) = \ell(xy).
	\end{align*}
	It follows from \cite{Tits69} that the symbols $\{ \delta_x ~\vert~ x \in W\}$ are a basis of $H$, which we call the standard basis. 
	Observe that $\delta_s^{-1} = \delta_s + v-v^{-1}$, so each $\delta_x$ is invertible. 
	\\
	\par 
	The Hecke algebra $H$ has a unique $\Z$-linear involution satisfying $\overline{v}=v^{-1}$ and $\overline{\delta_x} = \delta_{x^{-1}}^{-1}$ for each $x \in W$. 
	It is a classical result of Kazhdan and Lusztig \cite{KL79} that $H$ has a unique basis $\{b_x ~\vert~ x\in W \}$, called the \ldef{Kazhdan-Lusztig basis}, satisfying $\overline{b_x} = b_x$ and $b_x \in \delta_x + \sum_{y< x} v\Z[z] \delta_y$ for each $x \in W$. 
	We adopt the normalisation of \cite{Soe97}, where $b_s = \delta_s + v$ for each $s \in S$. 
	\\
	\par 
	For each $x,y \in W$, the \ldef{Kazhdan-Lusztig polynomials} $h_{y,x}$ and \ldef{inverse Kazhdan-Lusztig polynomials} $h^{y,x}$ are defined so that the following equalities hold:
	\begin{align*}
		b_x = \sum_y h_{y,x} \delta_y,
		&&
		\delta_x = \sum_{y} (-1)^{\ell(x)-\ell(y)} h^{y,x} b_y.
	\end{align*}
	That is, they arise as the change of basis coefficients between the standard and Kazhdan-Lusztig bases of $H$.
	Define $\mu(y,x) \in \Z$ to be the coefficient of $v$ in $h_{y,x}$. 
	It is a celebrated result of Elias and Williamson \cite{EW14} that $h_{y,x}, h^{y,x} \in  \Z_{\geq 0 }[v]$ and consequently $\mu(y,x) \in \Z_{\geq 0 }$. 
	\\
	\par
	We conclude with the following well-known, combinatorial lemma. 
	\begin{lem}
	\label{Lemma: bs action}
		Fix $x \in W$ and $s \in S$.  
		If $x<xs$, then $b_x b_s = b_{xs} + \sum_{ys<y} \mu(y,x) b_y$. 
		Alternatively, if $xs<x$, then $b_xb_s = (v+v^{-1})b_x$.
	\end{lem}
	
	\begin{proof}
		The first claim is immediate from the fact that 
		\begin{align*}
			\delta_x b_s 
			=
			\begin{cases}
				\delta_{xs} + v \delta_x
				& \text{if } xs>x,
				\\
				\delta_{xs} + v^{-1} \delta_x
				& \text{if } xs<x,
			\end{cases}
		\end{align*}
		and the requirement that $h_{y,x} \in v\Z[v]$ if $y<x$. 
		The second follows from a simple induction on $\ell(x)$ and the fact that $b_s^2 = (v+v^{-1})b_s$. 
	\end{proof}
	
\section{Hecke categories}

	Any category whose (split, triangulated, etc.) Grothendieck group is isomorphic to the Hecke algebra is called a Hecke category. 
	In this section we briefly recall various Hecke categories that will be the stage upon which our are arguments are set.

\subsection{Preliminaries}
\label{Section: Hecke cat prelim}

	Recall the notion of a realisation, as introduced in \cite[\S3.1]{EW16}.
	This is the data from which we will construct the diagrammatic categories in section \ref{Section: Hecke cat dia}.
	\\
	\par 
	Fix a Coxeter system $(W,S)$ and, for any simple reflections $s,t \in S$, let $m_{st}$ denote the order of $st$ in $W$. 
	Define the $\R$-vector space $\h := \bigoplus_{s \in S} \R \alphac_s$ and let $\h^* := \Hom_{\R}(\h, \R)$.
	Denote by $\pair{-,-}: \h \times \h^* \rightarrow \R$, the natural pairing. 
	For each $s \in S$ fix an element $\alpha_s \in \h^*$ so that 
	\begin{align*}
		\pair{\alphac_s , \alpha_t} = -2\cos\left(\frac{\pi}{m_{st}}\right)
	\end{align*}
	where $\pair{\alphac_s , \alpha_t} =-2$ when $m_{st} = \infty$. 
	Finally, we endow $\h$ with the structure of a $W$-module by requiring $s(\lambda) = \lambda- \pair{\lambda, \alpha_s}\alphac_s$ for each $s \in S$ and $\lambda \in \h$. 
	The data $(\h, \{ \alphac_s \}, \{ \alpha_s \})$ is called the \ldef{geometric realisation} of $(W,S)$. 
	It is a realisation in the sense of \cite[\S3.1]{EW16}. 
	\\
	\par 
	Set $R:= \Sym(\h)$, the symmetric algebra on the $\R$-vector space $\h$, which we consider as a graded $\R$-vector space where $\h$ is in degree $2$.
	Extend the $W$-module structure on $\h$ to a $W$-module structure on $R$ in the obvious way. 
	\\
	\par 
	\begin{rem}
		We consider the geometric realisation (over $\R$) because it is known to satisfy Soergel's conjecture, see \cite{EW14}. 
		The arguments in this note can easily be adapted to any balanced realisation satisfying Soergel's conjecture.   
	\end{rem}
	
\subsection{Diagrammatic categories}
\label{Section: Hecke cat dia}

	We now introduce various diagrammatic categories associated to the geometric realisation. 
	\\
	\par 
	The \ldef{diagrammatic Bott-Samelson category} $\sD^{\Z}_{\BS}(\h)$ is the following $\R$-linear, monoidal category which is enriched over $\Z$-graded vector spaces.
	Its objects are $B_{\underline{x}}$ where $\underline{x}$ is an expression in the simple reflections $s \in S$, i.e. $\underline{x}=(s_1 , \dots , s_k)$.
	The monoidal product is given by the concatenation of expressions $B_{\underline{x}} \star B_{\underline{y}} = B_{\underline{xy}}$. 
	Note that $B_{\emptyset}$ is the monoidal unit. 
	Morphisms in this category are $\R$-linear combinations of certain classes of planar diagram introduced in \cite{EW16}. 
	In particular, if distinct colours are assigned to each simple reflection in $S$, then morphisms are generated by the following vertices, their rotations, and homogeneous polynomials $f \in R$: 
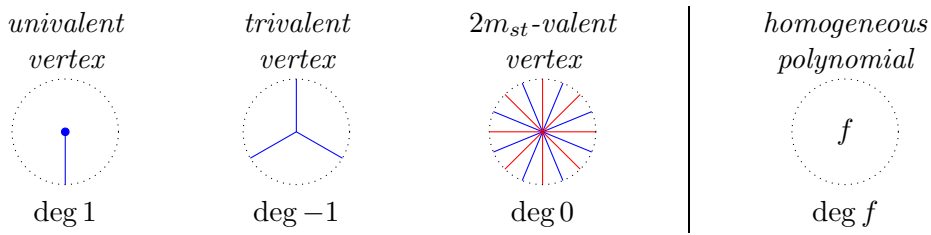
		\[
		\begin{array}{c c c c c c|c c}
			\begin{array}{c}
			\textit{univalent}
			\\
			\textit{vertex}
			\end{array}
		&~~~~& 
			\begin{array}{c}
			\textit{trivalent}
			\\
			\textit{vertex}
			\end{array}
		&~~~~&
			\begin{array}{c}
			2m_{st}\textit{-valent} 
			\\
			\textit{vertex} 
			\end{array}
		&~~&~~&
			\begin{array}{c}
			\textit{homogeneous}
			\\
			\textit{polynomial}
			\end{array}
		\\
			\begin{array}{c}
			\tikz[scale=0.7]
			{
			\draw[color=blue] (3,-1) to (3,0);
			\node[circle,fill,draw,inner sep=0mm,minimum size=1mm,color=blue] at (3,0) {};
			\draw[dotted] (3,0) circle (1cm);
			}
			\end{array}
		&&
			\begin{array}{c}
			\tikz[scale=0.7]
			{
			\draw[color=blue] (-30:1cm) -- (0,0) -- (90:1cm);
			\draw[color=blue] (-150:1cm) -- (0,0);
			\draw[dotted] (0,0) circle (1cm);
			}
			\end{array}
		&&
			\begin{array}{c}
			\tikz[scale=0.7]
			{
			\draw[color=blue] (0,0) -- (22.5:1cm);
			\draw[color=red] (0,0) -- (0:1cm);
			\draw[color=blue] (0,0) -- (67.5:1cm);
			\draw[color=red] (0,0) -- (45:1cm);
			\draw[color=blue] (0,0) -- (112.5:1cm);
			\draw[color=red] (0,0) -- (90:1cm);
			\draw[color=blue] (0,0) -- (157.5:1cm);
			\draw[color=red] (0,0) -- (135:1cm);			
			\draw[color=blue] (0,0) -- (-22.5:1cm);
			\draw[color=red] (0,0) -- (180:1cm);			
			\draw[color=blue] (0,0) -- (-67.5:1cm);
			\draw[color=red] (0,0) -- (-45:1cm);			
			\draw[color=blue] (0,0) -- (-112.5:1cm);
			\draw[color=red] (0,0) -- (-90:1cm);			
			\draw[color=blue] (0,0) -- (-157.5:1cm);
			\draw[color=red] (0,0) -- (-135:1cm);
			\draw[dotted] (0,0) circle (1cm);
			}
			\end{array}
		&&&
			\begin{array}{c}
			\tikz[scale=0.7]{
			\draw[dotted] (0,0) circle (1cm);
			\node at (0,0) {$f$};
			}
			\end{array}
		\\
			\deg 1
		&&
			\deg -1
		&&
			\deg 0
		&&&
			\deg f
		\end{array}
		\]
	where the  $2m_{st}$-valent vertices are coloured by pairs of $s,t \in S$ satisfying $m_{st} < \infty$. 
	The relations satisfied by these diagrams are stated in \cite[\S5.1]{EW16}.  
	If each homogeneous polynomial appearing in a diagram is a monomial, then the degree of the morphism is the sum of the degrees of the vertices (as given above) and degrees of the monomials. 
	We write $\Hom_{\sD^{\Z}_{\BS}(\h)}^{k} (B_{\underline{x}}, B_{\underline{y}})$ for the degree-$k$ morphisms from $B_{\underline{x}}$ to $B_{\underline{y}}$.
	\\
	\par 
	Categories enriched over $\Z$-graded vector spaces can equivalently be considered as categories endowed with an autoequivalence shift-functor $(1)$, see \cite[\S11.2.1]{EMTW20}.
	We denote the corresponding category with shift by $\sD^{(1)}_{\BS}(\h)$.  
	The objects in $\sD^{(1)}_{\BS}(\h)$ are $B_{\underline{x}}(n)$, where $\underline{x}$ is an expression and $n \in \Z$. 
	Morphisms in this category are then defined by
	\begin{align*}
		\Hom_{\sD^{(1)}_{\BS}(\h)} (B_{\underline{x}}(n), B_{\underline{y}}(m))
		:=
		\Hom_{\sD^{\Z}_{\BS}(\h)}^{m-n} (B_{\underline{x}}, B_{\underline{y}}).
	\end{align*}
	The monoidal structure is again given by concatenation of expressions in a manner that is compatible with the shift functor, namely $B_{\underline{x}}(n) \star B_{\underline{y}}(m) = B_{\underline{xy}}(n+m)$.
	For want of a better name, we will say that the object $B_{\underline{x}}(n)$ is in \ldef{Soergel degree} $n$.  
	\\
	\par 
	The \ldef{Elias-Williamson diagrammatic category} $\sD(\h)$ is the Karoubi envelope of the additive hull of $\sD^{(1)}_{\BS}(\h)$.
	It is a Krull-Schmidt, $\R$-linear, additive, monoidal category with shift-functor $(1)$.  
	Every indecomposable object in $\sD(\h)$ is isomorphic to an object of the form $B_x(n)$ where $x \in W$ and $n \in \Z$, see \cite[\S6.6]{EW16}. 
	Moreover, if $\underline{x} = (s_1, \dots ,s_k)$ is a reduced expression for $x$, i.e. $\ell(x)=k$, then $B_x$ occurs as a direct summand of $B_{\underline{x}}$ with multiplicity 1, and $B_x$ does not occur as a direct summand of any $B_{\underline{y}}$ where $\underline{y}$ is an expression of length less than $\ell(x)$. 
	\\
	\par 
	Let $[\sD(\h)]$ denote the split Grothendieck group of $\sD(\h)$. 
	The monoidal structure on $\sD(\h)$ endows $[\sD(\h)]$ with a ring structure, where $[B \star B'] = [B][B']$ for any objects $B,B'$. 
	Further, we endow $[\sD(\h)]$ with the structure of a $\Z[v,v^{-1}]$-module by imposing the relation $[B(1)] = v[B]$. 
	\\
	\par
	We conclude by recalling \ldef{Soergel's categorification theorem}.
	Namely, there is a unique isomorphism of $\Z[v,v^{-1}]$-algebras 
	\begin{align*}
		[\sD(\h)] \tilde{\longrightarrow} H
	\end{align*}
	induced by $[B_s]=b_s$, see \cite[\S6.6]{EW16}. 
	It will be incredibly important for the arguments in this note that geometric realisation satisfies Soergel's conjecture, see \cite{EW14}.
	In particular, if we identify $[\sD(\h)]$ with $H$ using the above isomorphism, then for each $x \in W$ we have 
	\begin{align}
	\label{Equation: Soergel conjecture}
		[B_x] = b_x.
	\end{align}
	Another consequence of Soergel's conjecture is that the graded ranks of Hom spaces can be determined using Soergel's Hom formula, see \cite[\S5]{Soe07}. 
	In particular, if $n \leq 0$ then
	\begin{align}
	\label{Equation: Soergels Hom Formula}
		\Hom_{\sD(\h)}(B_x , B_y(n))
		\cong
		\begin{cases}
			\R & x=y \text{ and } n=0, 
			\\ 
			0 & \text{otherwise}. 
		\end{cases}
	\end{align}
	If $n>0$ then $\dim_{\R} \Hom_{\sD(\h)}(B_x , B_y(n))$ is frequently non-zero and can be determined in terms of Kazhdan-Lusztig polynomials. 

\subsection{Mixed categories}
\label{Section: Hecke cat mixed}

	Finally, we introduce various mixed Hecke categories. 
	The identity $\delta_s = b_s -v$ implies there are no objects in $\sD(\h)$ which categorify the standard basis of $H$. 
	Thus, this extra level of complexity is necessary if one is interested in inverse Kazhdan-Lusztig polynomials. 
	\\
	\par 
	Let $\Ch(\h)$ denote the category of bounded cochain complexes in $\sD(\h)$.
	It is an additive, monoidal category. 
	Moreover, it has autoequivalence shift-functors $(1)$, inherited from $\sD(\h)$, and $[1]$, which is the cohomological shift functor. 
	Given a complex $C$ we write $C^i$ for the object in cohomological degree $i$. 
	The monoidal structure  on $\sD(\h)$ induces a monoidal structure on $\Ch(\h)$. 
	In particular given complexes $C$ and $D$, one defines the complex $C \otimes D$ by
	\begin{align*}
		(C \otimes D)^i := \bigoplus_{i=j+k} C^j \star D^k,
	\end{align*} 
	and the differential on $C^j \star D^k$ as $d_{C} \otimes 1 + (-1)^j \otimes d_D$. 
	We will abuse notation by identifying $\sD(\h)$ with the full subcategory of complexes concentrated in cohomological degree 0.
	\\
	\par 
	It will often be convenient to choose a splitting of each $C^i$. 
	Since $\sD(\h)$ is Krull-Schmidt, the multiplicity of each indecomposable summand in $C^i$ is independent of the choice of splitting, but the summands themselves are not. 
	No argument we use relies on a particular choice of splitting.
	Once a splitting has been chosen, the differential $d: C^i \rightarrow C^{i+1}$ is represented by a matrix of morphisms in $\sD(\h)$. 
	Following the notation of \cite[\S4.1]{Eli18}, an entry of this matrix is called a \ldef{summand of the differential}. 
	In other words, a summand of the differential is the restriction of the differential on $C^i$ so that its domain is a single indecomposable summand of $C^i$ and its codomain is a single indecomposable summand of $C^{i+1}$. 
	\\
	\par
	A complex $C$ in $\Ch(\h)$ is said to be \ldef{minimal} if $C$ is has no contractible summands.
	Equivalently, no summand of the differential is an isomorphism, see \cite[\S4.1]{Eli18}. 
	Given a complex $D$, a complex $C$ is said to be a \ldef{minimal subcomplex} of $D$ if: $C$ is a subcomplex of $D$; $C$ is minimal; and, the inclusion $C \hookrightarrow D$ is a homotopy equivalence. 
	Every complex $D$ has a minimal complex $D^{\min}$, and this complex is unique up to isomorphism in $\Ch(\h)$, see \cite[\S6.1]{EW14}.
	\\
	\par 
	We define the \ldef{mixed Hecke category}\footnote{This category is sometimes called the bi-equivariant category, as in \cite{ARV20}.} to be $D^{\mix}(\h) := K^b(\sD(\h))$. 
	That is, it is the bounded, homotopy category of $\sD(\h)$. 
	It is triangulated and monoidal, with monoidal structure defined as above, and has shift-functors $(1)$ and $[1]$.  
	\\
	\par 
	Consider the full subcategories $\cT^{\geq 0}$ and $\cT^{\leq 0}$ of $D^{\mix}(\h)$ whose objects are
	\begin{align*}
			\cT^{\geq 0}
			:=
			\{ C \in D^{\mix}(\h) 
			~\vert~  
			C^{\min, i} \in \langle B_x (k) ~\vert~ x \in W,~ k \leq i  \rangle_{\oplus} \},
		\\
			\cT^{\leq 0}
			:=
			\{ C \in D^{\mix}(\h) 
			~\vert~ 
			C^{\min, i} \in \langle B_x (k) ~\vert~ x \in W,~ k \geq i  \rangle_{\oplus} \},
	\end{align*}
	where $\langle B_x (k) ~\vert~ x \in W \rangle_{\oplus}$ denotes the full, additive, non-monoidal subcategory of $\sD(\h)$. 
	The pair $(\cT^{\leq 0}, \cT^{\geq 0})$ defines a $t$-structure on $D^{\mix}(\h)$ called the \ldef{perverse $t$-structure}, see \cite[\S6.3]{EW14}. 
	The heart of this $t$-structure is the \ldef{mixed perverse Hecke category} $\Perv(\h)$. 
	The objects in $\Perv(\h)$ can be explicitly characterised in terms of their minimal complexes, namely
	\begin{align}
	\label{Equation: Perverse}
		\Perv(\h) 
		:= 
		\Bigl\{ C \in D^{\mix}(\h) ~\Big\vert~ C^{\min,i} \cong \bigoplus_{x \in W} B_x(i)^{\oplus m_{x,i}} \Bigr\}
	\end{align}
	for some integers $m_{x,i}$. 
	In particular, if $C$ is a cochain complex where the object in cohomological degree $i$ is exclusively in Soergel degree $i$, then $C$ is a minimal complex whose homotopy equivalence class is an object in $\Perv(\h)$.
	The category $\Perv(\h)$ is endowed with a $t$-exact shift-functor $\tate{1} = [1](-1)$. 
	Moreover, up to isomorphism, every simple object is of the form $B_x \langle n\rangle$ for $x \in W$ and $n \in \Z$. 
	\\
	\par 
	As before we write $[\Perv(\h)]$ for the split Grothendieck group of $\Perv(\h)$.
	We conclude by noting that Soergel's categorification theorem extends to an isomorphism of $\Z[v,v^{-1}]$-modules:
	\begin{align*}
		[\Perv(\h)] \tilde{\longrightarrow} H
	\end{align*}
	where, for any $C \in \Perv(\h)$, we have 
	\begin{align*}
		[C] \longmapsto \sum_{i \in \Z} (-1)^{i} [C^{\min,i}]. 
	\end{align*}
	In particular, $[B_x \langle 1 \rangle] = -v^{-1}b_x$.

\section{Minimal complexes and injective morphisms}

	Having introduced the mixed perverse Hecke category $\Perv(\h)$, we now formulate a numerical condition on the number of indecomposable summands appearing in the minimal complexes of objects that is necessarily satisfied if there is an injective morphism between these objects. 
	\\
	\par 
	Fix a minimal complex $C^{\min}$ for each $C$ in $\Perv(\h)$.  
	As previously noted, $C^{\min}$ is unique up to isomorphism \cite[\S6.1]{EW14}.
	Typically, working with representatives of homotopy classes will not be functorial as composition will only be defined up to chain homotopy. 
	Fortunately, the validity of Soergel's conjecture ensures there are no non-trivial chain homotopies between minimal complexes. 
	Indeed, we have:
	
	\begin{lem}
	\label{Lemma: No homotopies}
		For any objects $C$ and $D$ in $\Perv(\h)$ we have 
		\begin{align*}
			\Hom_{\Perv(\h)} (C,D) 
			\cong 
			\Hom_{\Ch(\h)}(C^{\min},D^{\min}).
		\end{align*}
	\end{lem}

	\begin{proof}
	Since the geometric realisation satisfies Soergel's conjecture, Soergel's Hom formula (\ref{Equation: Soergels Hom Formula}) implies that for any two objects $B,B'$ in Soergel degree 0, we have $\Hom_{\sD(\h)}(B(i), B'(i-1)) = 0$. 
	It follows from the characterisation of $\Perv(\h)$ in (\ref{Equation: Perverse}) that any chain homotopy is of the form:
			\[
			\begin{tikzcd}[row sep=large, column sep = large]
				\phantom{C^{\min}}\dots \arrow[r] 
			&
				C^{\min , i-1 }
				\arrow[r] \arrow[d, shift right , "\phantom{x} f_{i-1}" right] 
				\arrow[d, shift left, "g_{i-1} \phantom{x}" left]  
			&  
				C^{\min , i }
				\arrow[r] \arrow[d, shift right ,"\phantom{x}f_i" right] 
				\arrow[d, shift left, "g_{i} \phantom{x}" left]  
				\arrow[dl, "0" above] 
			&
				C^{\min , i+1}
				\arrow[r] \arrow[d, shift right , "\phantom{x} f_{i+1}" right] 
				\arrow[d, shift left, "g_{i+1} \phantom{x}" left]  
				\arrow[dl, "0" above] 
			&
				\dots 
			\\
				\phantom{C^{\min}} \dots 
				\arrow[r] 
			&
				D^{\min , i-1 } 
				\arrow[r]  
			&  
				D^{\min , i}
				\arrow[r] 
			&
				D^{\min , i+1}
				\arrow[r] 
			&
				\dots 
			\end{tikzcd}
			\]
			so $f=g$. 
			Hence the claim. 
	\end{proof}

	Thus, taking minimal complexes defines a functor $(-)^{\mix}: \Perv(\h) \rightarrow \Ch(\h)$.
	Moreover, the following is immediate from Lemma \ref{Lemma: No homotopies}.  
	
	\begin{cor}
	\label{Corollary: fully faithful}
		The functor $(-)^{\mix}: \Perv(\h) \rightarrow \Ch(\h)$ is fully faithful. 
	\end{cor}
	
	Let us now fix some notation regarding the multiplicities of indecomposable summands in minimal complexes. 
	For any object $C$ in $\Perv(\h)$ recall that $C^{\min,i}$ denotes the object in $\sD(\h)$ in cohomological degree $i$ in the minimal complex $C^{\min}$. 
	By the Krull-Schmidt property of $\sD(\h)$ and the characterisation of the perverse $t$-structure in (\ref{Equation: Perverse}) there is a (non-canonical) splitting 
	\begin{align*}
		C^{\min,i} \cong \bigoplus_{y \in W} B_y(i)^{\oplus m_{y,C}^i} 
	\end{align*}
	where $m_{y,C}^i \in \Z_{\geq 0}$. 
	However, the Krull-Schmidt property implies  the multiplicities $m_{y,C}^i$ are independent of the choice of splitting and the choice of $C^{\min}$. 
	Thus, we define $m_{y,C} \in \Z_{\geq 0}[v,v^{-1}]$ as
	\begin{align*}
		m_{y,C} := \sum_{i \in \Z} m^i_{y,C} ~ v^i. 
	\end{align*}
	Finally, we define a partial order $\preceq$ on $\Z_{\geq 0}[v,v^{-1}]$ where $p \preceq q$ if $q-p \in \Z_{\geq 0}[v,v^{-1}]$.
	That is, $\preceq$ denotes \ldef{coefficient-wise inequality} on Laurent polynomials.
	\\
	\par 
	We now show that injective morphisms between perverse objects give rise to polynomials satisfying a monotonicity property. 
	\begin{prop}
	\label{Proposition: Monotonicity}
		Suppose there is an injective morphism $C \hookrightarrow D$ in $\Perv(\h)$, then $m_{y,C} \preceq m_{y,D}$ for each $y \in W$. 
	\end{prop}
	
	\begin{proof}
		We begin by noting that for any objects $C$ and $D$ in $\Perv(\h)$, Soergel's Hom formula implies
		\begin{align*}
			\Hom_{\sD(\h)}(C^{\min, i} , D^{\min,i})
			&\cong
			\Hom_{\sD(\h)}\left( \bigoplus_{y} B_y(i)^{\oplus m^i_{y,C}} , \bigoplus_{z} B_z(i)^{\oplus m^i_{z,D}} \right)
			\\&\cong
			\bigoplus_{y \in W} \Hom_{\sD(\h)}\left( B_y^{\oplus m^i_{y,C}} ,  B_y^{\oplus m^i_{z,D}} \right)
			\\&\cong
			\bigoplus_{y \in W}  \text{Mat}_{m^i_{y,C} \times m^i_{y,D}} (\R)
		\end{align*}
		where Mat${}_{p \times q}(\R)$ denotes the space of $p\times q$ matrices over $\R$.
		In particular, the full, additive, non-monoidal subcategory $\sD(\h,i):= \langle B_x(i) ~\vert~ x \in W  \rangle_{\oplus}$ of objects in a fixed Soergel degree is abelian, and equivalent to $W$-graded $\R$-vector-spaces, where $W$ is only considered as a set. 
		Note that $\sD(\h,i)$ can also be considered as the non-full subcategory of Soergel bimodules in Soergel degree $i$, where one only considers degree 0 morphisms.
		\\
		\par
		Let $f: C \rightarrow D$ be an injective morphism in $\Perv(\h)$ and assume for a contradiction there exists some $y \in W$ such that $m_{y,C} \not\preceq m_{y,D}$. 
		By Corollary \ref{Corollary: fully faithful} we have a well-defined chain map $f^{\min} : C^{\min} \rightarrow D^{\min}$. 
		Denote by $f^{\min,i} : C^{\min,i} \rightarrow D^{\min,i}$ the induced map in $\sD(\h,i)$. 
		By assumption, there is some $i \in \Z$ such that $m^i_{y,C} > m^i_{y,D}$.
		Hence, every morphism in $\text{Mat}_{m^i_{y,C} \times m^i_{y,D}}(\R)$  has non-zero kernel.
		In particular $\ker(f^i) \neq 0$.
		\\
		\par 
		Now consider the cochain complex $K(f)^{\min}$, which is defined so that $K(f)^{\min, i} := \ker(f^i)$ and whose differential is induced by the differential on $C^{\min}$. 
		Since $C^{\min}$ is minimal, no summand of the differential is an isomorphism, so the same is true of $K(f)^{\min}$.
		Consequently, $K(f)^{\min}$ is a minimal complex which is non-zero, so it is not homotopic to the zero complex.
		By construction $K(f)^{\min,i}$ is in Soergel degree $i$, so the homotopy equivalence class $K(f)$ of $K(f)^{\min}$ is in $\Perv(\h)$.  
		As $\ker(f^i)$ injects into $C^{\min,i}$ we have a non-zero morphism $K(f)^{\min} \rightarrow C^{\min}$. 
		By construction the composition $K(f)^{\min} \rightarrow C^{\min} \rightarrow D^{\min}$ is zero.
		Hence Corollary \ref{Corollary: fully faithful} implies $K(f)$ is a subobject of the kernel of $f$, which contradicts the fact that $f: C \rightarrow D$ is injective. 
	\end{proof}

\section{Rouquier complexes and monotonicity}

	If mixed Hecke categories were the stage upon which our arguments are set, then Rouquier complexes are the cast. 
	They are objects which categorify the standard basis of $H$. 
	We begin by showing the existence of injective morphisms between certain Rouquier complexes. 
	We then deduce Theorem \ref{Theorem: monotonicity} as a result.

\subsection{Injective morphisms between Rouquier complexes}
	
	For any simple reflection $s \in S$, we define objects $F_s$ and $F_s^{-1}$ in $D^{\mix}(\h)$ whose minimal complexes are of the form
	\begin{align}
		\label{Equation: Rouquier complex}
		\begin{array}{ccccccccccc}
			F_{\textcolor{blue} s}
		&
			:
		&
			0
		&
			\longrightarrow
		&
			0
		&
			\longrightarrow
		&
			B_{\textcolor{blue} s}
		&
			\overset{
			\begin{array}{c}
			\tikz[scale=0.3]
			{
			\draw[color=blue] (3,-1) to (3,0);
			\node[circle,fill,draw,inner sep=0mm,minimum size=1mm,color=blue] at (3,0) {};
			}
			\end{array}
			}{\longrightarrow}
		&
			B_{\id}(1)
		&
			\longrightarrow
		&
			0
		\\\\
			F_{\textcolor{blue} s}^{-1}
		&
			: 
		&
			0
		&
			\longrightarrow
		&
			B_{\id}(-1)
		&
			\overset{
			\begin{array}{c}
			\tikz[scale=0.3]
			{
			\draw[color=blue] (3,1) to (3,0);
			\node[circle, fill, draw, inner sep=0mm, minimum size=1mm, color=blue] at (3,0) {};
			}
			\end{array}
			}{\longrightarrow}
		&
			B_{\textcolor{blue} s}
		&
			\longrightarrow
		&
			0
		&
			\longrightarrow
		&
			0
		\end{array}
	\end{align}
	where each $B_s$ is in cohomological degree 0. 
	These were considered by Rouquier in the context of Soergel bimodules \cite[\S9]{Rou06}.
	Fix a reduced expression $\underline{x} = (s_1, \dots , s_k)$ for $x \in W$.
	We define the \ldef{Rouquier complexes} $\Delta_x$ and $\nabla_x$ in $D^{\mix}(\h)$ as 
	\begin{align*}
		\Delta_x := F_{s_1} \otimes \dots \otimes F_{s_k}
		&&
		\nabla_x := F_{s_1}^{-1} \otimes \dots \otimes F_{s_k}^{-1}.
	\end{align*} 
	Rouquier showed that if $(s_1, \dots , s_k)$ and $(s_1', \dots , s_k')$ are distinct reduced expressions for $x$, then $F_{s_1} \otimes \dots \otimes F_{s_k} \cong F_{s_1'} \otimes \dots \otimes F_{s_k'}$ in $D^{\mix}(\h)$, see   \cite[\S9.2.3]{Rou06}.
	Hence $\Delta_x$ is independent of the choice of reduced expression. 
	Similar statements hold for $\nabla_x$. 
	Note that if $\underline{x}= \emptyset$, then we define $\Delta_{\id} \cong \nabla_{\id} \cong B_{\id}$, which is the monoidal unit of $D^{\mix}$.
	Moreover, it is easy to show \cite[\S9.2.4]{Rou06} that 
	\begin{align}
	\label{Equation: invertible}
		F_s \otimes F_s^{-1} \cong B_{\id} \cong F^{-1}_{s} \otimes F_s.   
	\end{align}
	Consequently one has $\Delta_x \otimes \nabla_{x^{-1}} \cong B_{\id}$. 
	\\
	\par 
	Finally, it is important for all that follows that $\Delta_x$ and $\nabla_x$ are perverse for each $x \in W$ by \cite[\S6.5]{EW14}.
	\\
	\par 
	The following Lemma is essentially a repackaging of various arguments in \cite[\S6.4]{EW14}.
	
	\begin{lem}
	\label{Lemma: left t-exact}
		For each $w \in W$, the functors  $\Delta_w \otimes (-)$ and $(-) \otimes \Delta_w$ are left $t$-exact.
	\end{lem}
	\begin{proof}
		We prove the claim for $\Delta_w \otimes (-)$.
		The proof for $(-) \otimes \Delta_w$ is similar.  
		Observe that if $sw>w$ we have $\Delta_s \otimes \Delta_w \cong \Delta_{sw}$ by \cite[\S9.2.3]{Rou06}. 
		Hence it suffices to prove left $t$-exactness for $\Delta_s \otimes (-). $ 
		\\
		\par 
		Let $C$ an object in $D^{\mix}(\h)$ and $C^{\min}$ its minimal complex. 
		The `stupid' filtration 
		\begin{align*}
			(\sigma_{\geq i} C)^{\min} 
			:=  ~~~\dots 
			\longrightarrow 0 
			\longrightarrow C^{\min,i}
			\longrightarrow C^{\min, i+1} 
			\longrightarrow C^{\min, i+2} 
			\longrightarrow \dots
		\end{align*}
		gives rise to distinguished triangles 
		\begin{align*}
			\sigma_{\geq i+1} C
			\longrightarrow 
			\sigma_{\geq i} C
			\longrightarrow
			C^{\min,i}[-i]
			\longrightarrow
		\end{align*}
		for all $i \in \Z$. 
		Equation (\ref{Equation: invertible}) implies $\Delta_s \otimes (-)$ is an autoequivalence of $D^{\mix}(\h)$. 
		Hence we obtain distinguished triangles 
		\begin{align*}
			\Delta_s \otimes \sigma_{\geq i+1} C
			\longrightarrow 
			\Delta_s \otimes \sigma_{\geq i} C
			\longrightarrow
			\Delta_s \otimes C^{\min,i}[-i]
			\longrightarrow
		\end{align*}
		for all $i \in \Z$. 
		If both  $\Delta_s \otimes \sigma_{\geq i+1} C$ and $\Delta_s \otimes C^{\min,i}[-i]$ are objects in $\cT^{\geq0}$, then so is $\Delta_s \otimes \sigma_{\geq i} C$, by \cite[\S6.3]{EW14}.
		Hence left $t$-exactness of $\Delta_s \otimes (-)$ follows from showing $\Delta_s \otimes C^{\min,i}[-i]$ is in $\cT^{\geq 0}$, and induction on $i$.  
		\\
		\par 
		Suppose $B_y(j)$ is any indecomposable summand of $C^{\min,i}[-i]$. 
		If $sy > y$ then Soergel's conjecture (\ref{Equation: Soergel conjecture}) and Lemma \ref{Lemma: bs action} imply 
		\begin{align*}
			(\Delta_s \otimes B_y(j))^{\min} 
			\cong  ~~\dots 
			\longrightarrow 0 
			\longrightarrow \left(B_{sy} \oplus \bigoplus_{sz<z} B_z^{\oplus\mu(z,y)}\right) (j)
			\longrightarrow B_y(j+1) 
			\longrightarrow 0
			\longrightarrow \dots
		\end{align*} 
		Hence $\Delta_s \otimes B_y(j)$ is in $\cT^{\geq 0}$. 
		Alternatively, if $sy<y$ then Soergel's conjecture (\ref{Equation: Soergel conjecture}) and Lemma \ref{Lemma: bs action} imply
		\begin{align*}
			\Delta_s^{\min} \otimes B_y(j) 
			\cong  ~~\dots 
			\longrightarrow 0 
			\longrightarrow B_y(j-1)\oplus B_y(j+1)
			\longrightarrow B_y(j+1) 
			\longrightarrow 0
			\longrightarrow \dots
		\end{align*} 
		Since $\Delta_s \otimes(-)$ is an autoequivalence and $B_y(j)$ is indecomposable, it follows $\Delta_s \otimes B_y(j)$ is indecomposable.
		Indecomposability and Soergel's Hom formula (\ref{Equation: Soergels Hom Formula}), imply that the subcomplex $ B_y(j+1) \rightarrow  B_y(j+1)$ has a differential which is an isomorphism, so is contractible. 
		Hence
		\begin{align*}
			(\Delta_s \otimes B_y(j))^{\min} 
			\cong  ~~\dots 
			\longrightarrow 0 
			\longrightarrow B_y(j-1)
			\longrightarrow 0
			\longrightarrow \dots
		\end{align*}
		and $\Delta_s \otimes B_y(j)$ is in $\cT^{\geq 0}$. 
		Hence the claim. 
	\end{proof}
	
	We can now show the existence of injective morphisms between Rouquier complexes. 
	
	\begin{prop}
	\label{Proposition: Injectivity}
		Suppose $y\leq x$, then there exists an injective morphism $\Delta_{y} \langle \ell(y) \rangle \hookrightarrow \Delta_x \langle \ell(x) \rangle$.
	\end{prop}
	
	\begin{proof}
		Recall that the Bruhat order is generated by relations of the form $x' < x$ where: $x=usw$; $x'=uw$; $s \in S$; $\ell(x)= \ell(u)+ \ell(s)+ \ell(w)$; and $\ell(x') = \ell(u) + \ell(w)$.  
		Hence it suffices to show that there is an injective morphism $\Delta_{x'} \langle -1 \rangle \hookrightarrow \Delta_x$, as the desired morphism can then be obtained by composition.  
		\\
		\par 
		It is easy to show that, for each $s \in S$, there is an exact sequence 
		\begin{align*}
			\Delta_{\id}\langle -1\rangle \hookrightarrow \Delta_s  \twoheadrightarrow \Delta_{\id} \otimes B_s
		\end{align*}
		in $\Perv(\h)$, which gives rise to a distinguished triangle 
		\begin{align*}
			\Delta_{\id} \langle -1 \rangle  
			\longrightarrow 
			\Delta_s 
			\longrightarrow
			\Delta_{\id} \otimes B_s 
			\longrightarrow
		\end{align*}
		in $D^{\mix}(\h)$.
		Since we must necessarily have $u<us$, applying the functor $\Delta_u \otimes (-)$ to the preceding distinguished triangle gives the distinguished triangle
		\begin{align*}
			\Delta_{u} \langle -1 \rangle 
			\longrightarrow 
			\Delta_{us}
			\longrightarrow
			\Delta_{u} \otimes B_s 
			\longrightarrow.
		\end{align*}
		By \cite[\S6.5]{EW14} both $\Delta_{u}$ and $\Delta_{us}$ are perverse, so $\Delta_{u} \otimes B_s$ is an object in $\cT^{\leq 0}$. 
		By Lemma \ref{Lemma: left t-exact}, $\Delta_u \otimes (-)$ is left $t$-exact, so $\Delta_{u} \otimes B_s$ is an object in $\cT^{\geq 0}$.
		Hence $\Delta_{u} \otimes B_s$ is perverse.
		\\
		\par 
		Again, applying the functor  $(-) \otimes \Delta_{w}$, gives rise to a distinguished triangle 
		\begin{align*}
			\Delta_{x'} \langle -1 \rangle 
			\longrightarrow 
			\Delta_{x}
			\longrightarrow
			\Delta_{u} \otimes B_s \otimes \Delta_{w}
			\longrightarrow,
		\end{align*}
		where we have used the length conditions on $x$ and $x'$. 
		By an argument identical to above, we see $\Delta_{x'} \langle -1 \rangle$, $\Delta_{x}$ and $\Delta_{u} \otimes B_s \otimes \Delta_{w}$ are all perverse. 
		Hence we obtain an exact sequence
		\begin{align*}
			\Delta_{x'}\langle -1\rangle 
			\hookrightarrow \Delta_{x} 
			 \twoheadrightarrow \Delta_{u} \otimes B_s \otimes \Delta_{w}
		\end{align*}
		from which we obtain the desired injection. 
	\end{proof}

	\begin{rem}
		In \cite{ARV20} the authors construct an analogue of $\Perv(\h)$ when $\h$ is a realisation which need not satisfy Soergel's conjecture. 
		In this setting they show the existence of injective morphisms $\Delta_y \langle \ell(y) \rangle \hookrightarrow \Delta_x \langle \ell(x) \rangle$ when $y<x$, see \cite[\S8.2]{ARV20}. 
		From this, one can deduce that for each simple object $L_z \langle i \rangle$ in $\Perv(\h)$, we have $[\Delta_y \langle \ell(y) \rangle : L_z \langle i \rangle] \leq [\Delta_x \langle \ell(x) \rangle : L_z \langle i \rangle]$. 
		However, when $\h$ does not satisfy Soergel's conjecture it will necessarily be the case that there are some $z \in W$ for which $L_z \not \cong B_z$. 
		In particular, inverse $p$-Kazhdan-Lusztig polynomials need not satisfy the analogous monotonicity property. 
	\end{rem}

\subsection{The first half of the first theorem.}

	If we wish to prove the statement in Theorem \ref{Theorem: monotonicity} pertaining to inverse Kazhdan-Lusztig polynomials, it suffices, by Propositions \ref{Proposition: Monotonicity} and \ref{Proposition: Injectivity}, to show 
	\begin{align*}
		m_{y,\Delta_x} = h^{y,x}
	\end{align*} 
	for each $x,y \in W$. 
	This statement is asserted without proof in \cite[\S6.5]{EW14}. 
	For the sake of completeness, we provide a proof. 
	\\
	\par 
	We first require the following parity property of inverse Kazhdan-Lusztig polynomials. 
	
	\begin{lem}
	\label{Lemma: Parity}
		For any $x,y \in W$ we have $h^{y,x} \in v^{\ell(xy)}\Z[v^{-2},v^2]$.
	\end{lem}
	\begin{proof}
		We prove the claim by induction on $\ell(x)$. 
		Note that it is trivially true for $x=\id$, as $h^{\id,\id}=1$. 
		Now suppose the claim holds for $x$, and $s \in S$ is chosen so that $xs>x$.
		Then, by assumption, one has
		\begin{align*}
			\delta_{xs}  
			= \delta_x (b_s -v)
			&\in 
			\sum_y v^{\ell(xy)}\Z[v^{-2},v^2] b_y (b_s - v)
		\end{align*}
		Now note that if $ys<y$ then Lemma \ref{Lemma: bs action} implies
		\begin{align*}
			v^{\ell(xy)}\Z[v^{-2},v^2] b_y (b_s-v) 
			=
			v^{\ell(xy)-1}\Z[v^{-2},v^2] b_y 
			=
			v^{\ell(xsy)}\Z[v^{-2},v^2] b_y.
		\end{align*}
		Alternatively, if $ys>y$ then Lemma \ref{Lemma: bs action} implies
		\begin{align*}
			v^{\ell(xy)}\Z[v^{-2},v^2] b_y (b_s-v) 
			=
			v^{\ell(xy)}\Z[v^{-2},v^2] \left(b_{ys}-vb_y +\sum_{zs<z}\mu(z,y)b_z\right).
		\end{align*} 
		The claim then follows from the fact that: $\ell(xy)\equiv \ell(xsys) \mod 2$; and, if $zs<z$ then $\ell(xy)\equiv \ell(xsz) \mod 2$. 
	\end{proof}
	
	We now show that inverse Kazhdan-Lusztig polynomials encode the Jordan-H\"{o}lder multiplicities of Rouquier complexes. 
	
	\begin{prop}
		For any $x,y \in W$ we have $m_{y,\Delta_x} = h^{y,x}$.
	\end{prop}
	\begin{proof}
		We first show that $[\Delta_x] = \delta_x$ for each $x  \in W$. 
		We proceed by induction on $\ell(x)$. 
		By definition we have $\Delta_{\id} \cong B_{\id}$, so $[\Delta_{\id}] = b_{\id} = 1$. 
		Now observe that since $\Delta_s = F_s$, the description of the complex in Equation (\ref{Equation: Rouquier complex}) implies $[\Delta_s] = [B_s] - v[B_{\id}] = b_s - v$. 
		As $b_s = \delta_s +v$, it follows $[\Delta_s]= \delta_s$. 
		Now assume that $[\Delta_x]= \delta_x$ and $s \in S$ is chosen such that $xs>x$ then
		\begin{align*}
			[\Delta_{xs}] 
			=
			[F_{s_1} \otimes \dots  \otimes F_{s_k} \otimes F_s]
			=
			[\Delta_{x} \otimes \Delta_s]
			=
			[\Delta_{x}][\Delta_s]
			=
			\delta_x \delta_s
			= 
			\delta_{xs}. 
		\end{align*}
		Hence the first claim. 
		\\
		\par 
		Now observe that by the validity of Soergel's conjecture (\ref{Equation: Soergel conjecture}), we have 
		\begin{align*}
			[\Delta_{x}]
			= 
			\sum_{i}\sum_{y} (-1)^i ~ m_{y, \Delta_x}^i v^i [B_y]
			=
			\sum_{y} \left(\sum_{i} (-1)^i ~ m_{y, \Delta_x}^i v^i \right) b_y,
		\end{align*}
		where we may exchange the order of summation as objects in $D^{\mix}(\h)$ are bounded. 
		Then the parity property of Lemma \ref{Lemma: Parity} implies
		\begin{align*}
			(-1)^{\ell(x)-\ell(y)}h^{y,x}
			= 
			\sum_{i} (-1)^i ~ m_{y, \Delta_x}^i v^i 
			=
			(-1)^{\ell(x)-\ell(y)}m_{y, \Delta_x}
		\end{align*}
		which proves the claim. 
	\end{proof}

\section{Parabolic Kazhdan-Lusztig polynomials}

	We conclude by showing certain parabolic inverse Kazhdan-Lusztig polynomials, introduced by Deodhar in \cite{Deo87}, satisfy an analogous monotonicity property.  
	\\
	\par 
	Fix a Coxeter system $(W,S)$ and a subset $I \subseteq S$. 
	Let $H_I$ denote the Hecke algebra associated to the Coxeter system $(W_I, I)$, which we consider as a subalgebra of $H$. 
	We can endow $\Z[v,v^{-1}]$ with the structure of a right $H_I$-module in two distinct ways:  
	let $\triv^I$ denote the rank $1$ $H_I$-module where $\delta_s$ acts by $v^{-1}$ for each $s \in I$; and, $\sign^I$ denote the rank $1$ $H_I$-module where $\delta_s$ acts by $-v$ for each $s \in I$.
	The \ldef{spherical module} ${}^I M$ and \ldef{antispherical module} ${}^I N$ are the right $H$-modules 
	\begin{align*}
		{}^I M := \Ind_{H^I}^H \triv^I, 
		&&
		{}^I N := \Ind_{H^I}^H \sign^I.
	\end{align*}
	Denote by ${}^I W$ the set of minimal length coset representatives for $W_I \backslash W$ and $\leq$ the induced Bruhat order.  
	For either module we (abusively) denote by $\delta_x^I$ the element $1 \otimes_{H_I} \delta_x$.
	It is a well-known fact that $\{ \delta_x ^I ~|~ x \in {}^I W\}$ is a basis of each module, which we call the \ldef{standard basis}. 
	\\
	\par 
	The Kazhdan-Lusztig involution restricts to a $\Z$-linear involution of $\Z[v,v^{-1}]$. 
	Moreover, it is easy to check that this induces involutions on ${}^I M$ and ${}^I N$. 
	In \cite{Deo87} Deodhar showed ${}^IM$ and ${}^IN$ admit bases which are analogous to the Kazhdan-Lusztig basis of $H$. 
	Namely, the \ldef{spherical Kazhdan-Lusztig basis} $\{ c_x ~|~ x \in {}^I W \}$ of ${}^IM$ and  the \ldef{antispherical Kazhdan-Lusztig basis} $\{ d_x ~|~ x \in {}^I W \}$ of ${}^IN$ are the bases which are uniquely characterised by the properties:
	\begin{align*}
		\overline{c_x} = c_x ~~~\text{ and }~~~ &c_x \in \delta_x^{I} + \sum_{y<x} v\Z[v] \delta_y^I
		\\
		\overline{d_x} = d_x ~~~\text{ and }~~~ &d_x \in \delta_x^{I} + \sum_{y<x} v\Z[v] \delta_y^I 
	\end{align*}
	for each $x \in {}^I W$. 
	\\
	\par 
	As in the classical case, for each $x,y \in {}^I W$ the \ldef{spherical Kazhdan-Lusztig polynomials} $m_{y,x}$ and \ldef{spherical inverse Kazhdan-Lusztig polynomials} $m^{y,x}$ are defined so that the following equalities hold:
	\begin{align*}
		c_x = \sum_y m_{y,x} \delta_y^I,
		&&
		\delta_x^I = \sum_{y} (-1)^{\ell(x)-\ell(y)} m^{y,x} c_y.
	\end{align*}
	Similarly, the \ldef{antispherical Kazhdan-Lusztig polynomials} $n_{y,x}$ and \ldef{antispherical inverse Kazhdan-Lusztig polynomials} $n^{y,x}$ are defined so that:
	\begin{align*}
		d_x = \sum_y n_{y,x} \delta_y^I,
		&&
		\delta_x^I = \sum_{y} (-1)^{\ell(x)-\ell(y)} n^{y,x} d_y.
	\end{align*}
	
	Having established this notation, we can now finish the proof of Theorem \ref{Theorem: monotonicity}. 
	\begin{proof}[Proof of Theorem \ref{Theorem: monotonicity}]
		The claim for antispherical inverse Kazhdan-Lusztig polynomials follows immediately from the classical case and \cite[Proposition 3.7]{Soe07}. 
		In particular, the latter asserts that for an arbitrary Coxeter system $(W,S)$ and subset $I \subseteq S$, for any $x,z \in {}^I W$ we have an equality $n^{z,x} = h^{z,x}$. 
	\end{proof}
	
	\begin{rem}
		We conclude by noting that the spherical inverse Kazhdan-Lusztig polynomials $m^{y,x}$ do not satisfy a monotonicity property.
		This is evident even when $W$ is finite. 
		For example, when $W$ is type $\tA_n$ and $W_I$ is type $\tA_{n-1}$, the Bruhat order on ${}^I W$ is a chain, and monotonicity fails to hold for any consecutive $z < y < x$.
	\end{rem}

\section*{Acknowledgements}

	This note was part of the author's PhD thesis which was completed at the University of Sydney under the supervision of Geordie Williamson.
	I would like to thank Geordie for suggesting this problem.
	I would also like to thank Geordie and Noriyuki Abe for valuable comments on preliminary versions of this note.   
	The author was supported by the award of a Research Training Program scholarship.

%==========================================================================

%============================== BIBLIOGRAPHY ==============================

%==========================================================================
%==========================================================================

\end{document}